\documentclass[11pt,a4paper]{article}

\usepackage{amsfonts}
\usepackage{amsthm}
\usepackage{amsmath}
\usepackage{amssymb}
\usepackage{amscd}
\usepackage{mathrsfs}
\usepackage[mathscr]{eucal}
\usepackage{graphicx}
\usepackage{epstopdf} 
\usepackage{pict2e}
\usepackage{epic}
\usepackage{xcolor}
\usepackage{float}
\usepackage{mathtools}
\usepackage{centernot}

\usepackage{cite}
\usepackage{hyperref}
\hypersetup{colorlinks=true, urlcolor= black, linkcolor=black, citecolor=blue}

\usepackage[utf8]{inputenc}
\usepackage[T1]{fontenc}
\usepackage{lmodern}
\usepackage{dsfont}
\usepackage{indentfirst}

\usepackage{hyperref}   
\hypersetup{
    linktoc=page,
    linkcolor=blue,          
    citecolor=blue,        
    filecolor=blue,      
    urlcolor=black,
   colorlinks=true           
}

\numberwithin{equation}{section}

\theoremstyle{plain}
\newtheorem{Thm}{Theorem}[section]
\newtheorem{Lem}[Thm]{Lemma}

\newtheorem{Coro}[Thm]{Corollary}
\newtheorem{Prop}[Thm]{Proposition}

\theoremstyle{definition}

\newtheorem{Rem}[Thm]{Remark}
\newtheorem{Pb}[Thm]{Problem}

\usepackage[margin=2.5cm]{geometry}

\definecolor{darkgreen}{rgb}{0,0.6,0.05}

\newcommand{\connect}{\xleftrightarrow}
\newcommand\fg{\mathfrak{g}}

\let\qed=\QED

\renewcommand{\epsilon}{\varepsilon}

\newcommand{\Z}{\mathbb{Z}}

\def\H{\mathbb{H}}

\def\P{\mathbb{P}} 
\def\E{\mathbb{E}} 
\def\md{\mid}

\def\Bb#1#2{{\def\md{\bigm| }#1\bigl[#2\bigr]}}

\def\Eb{\Bb\E}

\def \p {{\partial}}

\def\<#1{\langle #1\rangle}

\def\bi{\begin{itemize}}  
\def\ei{\end{itemize}}
\def\bnum{\begin{enumerate}} 
\def\enum{\end{enumerate}}

\title{One-arm exponents of high-dimensional percolation revisited}

\begin{document}
	\author{Diederik van Engelenburg\footnotemark[1]\footnote{TU Wien, Wien, Austria, \url{diederik.engelenburg@tuwien.ac.at}}\:, Christophe Garban\footnotemark[2]\footnote{
Universit\'e Claude Bernard Lyon 1, CNRS UMR 5208, Institut Camille Jordan, 69622 Villeurbanne, France, and Courant Institute (NYU), New York, USA, \url{christophe.garban@gmail.com}}\:, Romain Panis\footnotemark[3]\footnote{Universit\'e Claude Bernard Lyon 1, Villeurbanne, France, \url{panis@math.univ-lyon1.fr}}\:,\\ Franco Severo\footnotemark[4]\footnote{CNRS and Sorbonne Universit\'e, Paris, France, \url{severo@lpsm.paris}}}

	\maketitle

	\abstract{We consider sufficiently spread-out Bernoulli percolation in dimensions ${d>6}$. We present a short and simple proof of the up-to-constants estimate for the one-arm probability in both the full-space and half-space settings. These results were previously established by Kozma and Nachmias and by Chatterjee and Hanson, respectively. Our proof improves upon the entropic technique introduced by Dewan and Muirhead, relying on a sharp estimate on a suitably chosen correlation length recently obtained by Duminil-Copin and Panis. This approach is inspired by our companion work \cite{vEGPS}, where we compute the one-arm exponent for several percolation models related to the high-dimensional Ising model.

}

	\section{Introduction}\label{sec:intro}
	
		In the context of \emph{finite-range} Bernoulli percolation on $\mathbb Z^d$ ($d\geq 2$), it is predicted that there is no infinite cluster at the critical point: $\theta(p_c):=\mathbb P_{p_c}[0\connect{}\infty]=0$. In the last fifty years, significant progress has been made regarding this conjecture. Kesten \cite{Kesten1980criticalproba} proved it in the case of two-dimensional (planar) nearest-neighbour percolation. In higher dimensions, the \emph{lace expansion} allowed to prove that $\theta(p_c)=0$ in dimensions $d>10$ for the nearest-neighbour case \cite{HaraSlade1990Perco,HaraDecayOfCorrelationsInVariousModels2008,FitznervdHofstad2017Percod10}, and in dimensions $d>6$ for sufficiently spread-out percolation \cite{HaraSladevdHofstad2003PercoSO} (see precise definition below). A more recent approach \cite{DumPan24Perco} provides an alternative proof in the latter setting. Proving that $\theta(p_c)=0$ in the ``intermediate'' dimensions $3\leq d\leq 6$ for either of the two aforementioned settings remains one of the most challenging problems in percolation theory. 
		
		The vanishing of $\theta(p_c)$ is equivalent to the convergence to $0$ of the sequence of \emph{one-arm probabilities} $\mathbb P_{p_c}[0\connect{}\Lambda_n^c]$, where $\Lambda_n:=[-n,n]^d\cap \mathbb Z^d$. According to Widom's scaling hypothesis \cite{Widom1965equation}, the behaviour of this sequence should be described by a critical exponent $\rho=\rho(d)>0$---the so-called \emph{one-arm exponent}---defined by:
		\begin{equation}\label{eq:intro one arm proba}
			\mathbb P_{p_c}[0\connect{}\Lambda_n^c]=n^{-\rho+o(1)},
		\end{equation}
where $o(1)$ tends to $0$ as $n$ tends to infinity.

In a seminal work, Kozma and Nachmias \cite{KozmaNachmias} proved that $\rho(d)=2$ whenever the critical two-point function satisfies $\mathbb P_{p_c}[0\connect{}x]\asymp |x|^{2-d}$ and $d>6$ (where $|\cdot|$ is the infinite norm on $\mathbb R^d$, and where $\asymp$ means that the two quantities are within a multiplicative constant of each other). This assumption is expected to hold in dimensions $d>6$, and has been verified in dimensions $d>10$ for the nearest-neighbour case \cite{HaraSlade1990Perco,HaraDecayOfCorrelationsInVariousModels2008,FitznervdHofstad2017Percod10}, and $d>6$ for sufficiently spread-out percolation \cite{HaraSladevdHofstad2003PercoSO,DumPan24Perco}. Let us also mention that with the breakthrough proof of conformal invariance of Smirnov \cite{SmirnovCardy2001}, it is possible to show that $\rho=\tfrac{5}{48}$ for site percolation on the triangular lattice \cite{lawler2002one}.

In this article, we focus on spread-out percolation in dimensions $d>6$. Building upon the results of \cite{DumPan24Perco}, we provide an alternative simple proof of the result of \cite{KozmaNachmias}.
We also compute (up to a multiplicative constant) the one-arm probability in the half-space setting, a result previously derived by Chatterjee and Hanson \cite{ChatterjeeHanson}. Our results rely on the entropic method of Dewan and Muirhead \cite{DewanMuirhead}, which we refine by using a more suitable notion of correlation length. We stress that our proofs do not rely on any result obtained through the use of lace expansion.
	
	\subsection{Definitions and statement of the results}	
	We consider spread-out percolation on $\Z^d$ with \emph{spread} parameter $L\geq1$. More precisely, let $\mathcal G=(V,E)$ be the graph with vertex set $V=\Z^d$ and edge set $E=E_L:=\{\{x,y\}\subset \Z^d:~|x-y|\leq L\}$. We write $x\sim y$ to say that $\{x,y\}\in E_L$. We denote by $\P_p$ the law of the random subgraph of $\mathcal G$ obtained by independently keeping (resp.~removing) each edge with probability $p$ (resp.~$1-p$). For $A,B,S\subset \Z^d$, we denote by $\{A \connect{S\:} B\}$ the event that there exists an open path from $A$ to $B$ which is fully contained in $S$. We denote by $\tau^S_p(x,y):=\mathbb P_{p}[x\connect{S\:}y]$ the restricted two-point function. When $S=\Z^d$, we omit it from the notation. 
	
	It is classical (see \cite{GrimmettPercolation1999}) that, when $d\geq 2$ and $L\geq 1$, the model undergoes a phase transition at a non-trivial critical point
	\begin{equation}
	p_c=p_c(d,L):=\inf\big\{p\in[0,1]: \P_p[0\connect{}\infty]>0\big\}.
	\end{equation}
Hara and Slade \cite{HaraSlade1990Perco} first highlighted the relevance of spread-out percolation. The \emph{universality hypothesis} asserts that the critical exponents of spread-out and nearest-neighbour percolation should match. The spread-out setting is convenient to compute these exponents as one can take the parameter $L$ large enough to apply perturbative techniques \cite{HaraSladevdHofstad2003PercoSO,DumPan24Perco}. 
	
Our first result concerns the study of the one-arm probability introduced in \eqref{eq:intro one arm proba}.

	\begin{Thm}\label{thm: main result} Let $d>6$. There exists $L_0\geq 1$ such that the following holds. For every $L\geq L_0$, there exist $c,C>0$ such that, for every $n\geq 1$, 
	\begin{equation}\label{eq:main_bulk}
	\frac{c}{n^2}\leq \P_{p_c}[0\connect{} \Lambda_n^c] \leq \frac{C}{n^2},
	\end{equation}		
	\end{Thm}

	As mentioned above, this result was first obtained by Kozma and Nachmias \cite{KozmaNachmias}. Their analysis is based on two fundamental inputs: 
	\begin{enumerate}
		\item[$(i)$] an up-to-constant estimate on the model's two point function at criticality: 
		\begin{equation}
			\tau_{p_c}(0,x)\asymp |x|^{2-d};
		\end{equation}
		\item[$(ii)$] a sharp estimate for the volume of cluster of the origin: letting $\mathcal C(0)$ denote the cluster of the origin, for every $n\geq 1$, 
		\begin{equation}\label{eq:input cluster volume}
			\mathbb P_{p_c}[|\mathcal C(0)|\geq n]\asymp n^{-1/2}.
		\end{equation} 
	\end{enumerate}
	The first input is very standard in the study of high-dimensional percolation. It relies on the ``black box'' which can either be the lace expansion or the approach of \cite{DumPan24Perco}. The second one $(ii)$ was derived via differential inequality techniques: the lower bound was derived in \cite{AizenmanBarsky1987sharpnessPerco} (see also \cite{DuminilTassionNewProofSharpness2016}), and the upper bound in \cite{BarskyAizenmanCriticalExponentPercoUnderTriangle1991} (see also \cite{hutchcroft2022derivation}). The upper bound in \eqref{eq:input cluster volume} holds under the so-called \emph{triangle condition}
	\begin{equation}\label{eq:triangle condition}
		\nabla(p_c):=\sum_{x,y\in \mathbb Z^d}\tau_{p_c}(0,x)\tau_{p_c}(x,y)\tau_{p_c}(y,0)<\infty
	\end{equation}
	Observe that the triangle condition is clearly satisfied if $(i)$ holds and if $d>6$. With these inputs, the proof of \cite{KozmaNachmias} consists in proving a renormalisation inequality on $\theta_n(p_c):=\mathbb P_{p_c}[0\connect{} \Lambda_n^c]$. To put it in a nutshell, they consider two cases: on top of the occurrence of the one-arm event, either $|\mathcal C(0)|\geq \varepsilon n^4$ or $|\mathcal C(0)|<\varepsilon n^4$. By $(ii)$ above, the first case has a probability $\lesssim 1/n^2$ which is of the good order. The handling of the second term is significantly harder and relies on the development of a refined theory of ``regularity of clusters''. Let us mention that the latter step was recently revisited and streamlined \cite{asselah2025capacity}, relying on a more geometric notion of regularity. A big advantage of the approach of Kozma and Nachmias is that these regularity properties are robust and useful beyond the context of one-arm exponent computations. Subsequent applications include \cite{vdHofstadSapoS14,ChatterjeeHanson,ChatterjeeHansonSosoe2023subcritical,chatterjee2025robust,asselah2025capacity}.
	 	
	Instead of relying on $(i)$ and $(ii)$, we use as an input some sharp near-critical estimates of the susceptibility $\chi(p)$ and the quantity $\varphi_p(S)$, introduced in \cite{DuminilTassionNewProofSharpness2016}. These estimates were proved in \cite{DumPan24Perco} for spread-out percolation, and are summarized in Theorem \ref{thm:INPUT} below.
	We stress that their proof relies  neither on the lace expansion nor on \cite{BarskyAizenmanCriticalExponentPercoUnderTriangle1991}.

The second result of this paper concerns the computation of the one-arm exponent in the half-space setting. Let $\mathbb H:=\{0,1,\ldots\}\times\mathbb Z^{d-1}$.

\begin{Thm}\label{thm: halfspace one arm} Let $d>6$. Let $L_0$ be given by Theorem \textup{\ref{thm: main result}}, and $L\geq L_0$. There exist $c,C>0$ such that, for $n\geq 1$,
	\begin{equation}\label{eq:main_half-space}
		\frac{c}{n^3} 
		\leq  \P_{p_c}[0\connect{\mathbb H\:}\Lambda_{n}^c]   \leq \frac{C}{n^3}.
	\end{equation}
\end{Thm}
This result was previously derived by Chatterjee and Hanson \cite{ChatterjeeHanson}. One important feature of their proof is that it heavily relies on  \cite{KozmaNachmias}: not only does it use their one-arm computation, by more fundamentally the whole theory of regularity that is developed there. However, let us mention that Chatterjee and Hanson go further by also computing the exponent of the two-point function restricted to the half-space.

Before moving to a precise explanation of our proofs, let us stress that the restriction to spread-out percolation is somewhat artificial in what follows. More precisely, if one could prove (e.g.\ for nearest-neighbour percolation) an analogue of Theorem \ref{thm:INPUT} below which does not rely on \cite{KozmaNachmias}, the whole strategy of this paper would go through.

\paragraph{Notations.} If $x=(x_1,\ldots,x_d)\in \mathbb R^d$, we let $|x|:=\max_{1\leq i\leq d}|x_i|$ denotes the $\ell^\infty$ norm of $x$. We let $\mathbf{e}_1:=(1,0,\ldots,0)\in \mathbb R^d$. Throughout the paper various $(d,L)$-dependent constant appear without playing any particular role. We will therefore use the following compact notation: we write $f\lesssim g$ (or $g\gtrsim  f$) if there exists a constant $C=C(d,L)\in (0,\infty)$ such that $f\leq Cg$.

\subsection{Strategy of the proof}\label{sec:introStrat}	
	
We will use the entropic bound of \cite{DewanMuirhead} in the precise formulation of the lemma below. 
\begin{Lem}[{\hspace{1pt}\cite[Proposition~1.21]{DewanMuirhead}}]\label{l.EB} For every $d\geq 2$ and $L\geq 1$, there exists a constant $C_0=C_0(d,L)>0$, such that for every set $\Lambda\subset \mathbb Z^d$ containing $0$, every event $\mathcal{A}$ measurable with respect to $\mathcal C_{\Lambda}(0):=\{x\in \mathbb Z^d:0\connect{\Lambda\:}x\}$, and every $\tfrac{p_c}{2}\leq p'< p\leq \tfrac{p_c+1}{2}$, we have
\begin{equation}\label{eq:EB}
\big|\mathbb{P}_{p}[\mathcal{A}] - \mathbb{P}_{p'}[\mathcal{A}]\big| \leq C_0 |p-p'| \sqrt{\max\{\mathbb{P}_{p}[\mathcal{A}], \mathbb{P}_{p'}[\mathcal{A}]\}
\mathbb{E}_{p'}[|\mathcal C_{\Lambda}(0)|]}.
\end{equation}
\end{Lem}
\begin{proof} We first assume that $\Lambda$ is finite and apply \cite[Proposition~1.21]{DewanMuirhead} to the standard exploration of the cluster of the origin in $\Lambda$. Since $p_c\in (0,1)$ one has $\tfrac{p_c+1}{2},\tfrac{p_c}{2}\in (0,1)$ so that the quantities $(p(1-p))^{-1/2}, (p'(1-p'))^{-1/2}$ (which appear in the entropic bound of \cite{DewanMuirhead}) are bounded away from infinity by a constant which only depends on $p_c$, and hence on $d,L$. The extension to the case $\Lambda$ infinite follows by approximation.
\end{proof}

We remark that \eqref{eq:EB} can be seen as a generalization of the classical differential inequality $\frac{d}{dp} \mathbb{P}_{p}[\mathcal{A}] \leq \frac{1}{p(1-p)} \sqrt{ \mathbb{P}_{p}[\mathcal{A}] \mathbb{E}_{p}[|\mathcal C_{\Lambda}(0)|]}$, proved in \cite{OS2007learning}.

In \cite{DewanMuirhead}, the authors applied the above result to obtain a bound on the one-arm probability of \eqref{eq:intro one arm proba}. Let us briefly explain their argument. Apply Lemma \ref{l.EB} to $\Lambda=\mathbb Z^d$, $\mathcal A:=\{0\connect{} \Lambda_n^c\}$, and $p=p_c$. Then, pick $p'$ as follows. First, notice that one may find\footnote{Since $\varphi_{p_c}(\Lambda_n)=p_c\sum_{\substack{x\in \Lambda_n, \:y\notin \Lambda_n, \:x\sim y}}\tau^{\Lambda_n}_{p_c}(0,x)\geq 1$ (see \cite{DuminilTassionNewProofSharpness2016}), one has $\mathbb P_{p_c}[0\connect{} \Lambda_n^c]\gtrsim n^{1-d}$. This justifies the existence of $\alpha$. } $\alpha=\alpha(d,L)>0$ such that, for every $n\geq 2$,
\begin{equation}
	\mathbb P_{p_c}[0\connect{}\Lambda_n^c]\geq \frac{1}{n^\alpha}.
\end{equation}
The idea is to pick $p'=p'(n)$ such that
\begin{equation}\label{eq:onearm drops at p' dm}
	\mathbb P_{p'}[0\connect{}\Lambda_n^c]\leq \frac{1}{2 n^\alpha}\leq \frac{1}{2}\mathbb P_{p_c}[0\connect{}\Lambda_n^c].
\end{equation}
To prove the existence of $p'$, the authors consider the \emph{correlation length} $\xi(p)$, defined as
	\begin{equation}
		\xi(p')^{-1}\coloneqq \lim_{n\to\infty} -\frac{1}{n}\log \tau_{p'}(0, n \mathbf{e}_1).
	\end{equation}
Combining the bound $\xi(p')\lesssim (p_c-p')^{-1/2}$ obtained in \cite{hara1990mean,HaraSlade1990Perco}, and the bound
\begin{equation}
	\mathbb P_{p'}[0\connect{} \Lambda_n^c]\leq \sum_{x\in  \Lambda_{n+L}\setminus \Lambda_{n}}\tau_{p'}(0,x)\lesssim n^{d-1}e^{-n/\xi(p')},
\end{equation}
(where the second inequality follows from the inequality $\tau_{p'}(0,x)\leq e^{-|x|/\xi(p')}$ \cite[Chapter~6.2]{GrimmettPercolation1999}) implies that \eqref{eq:onearm drops at p' dm} holds for $p'=p_c-\frac{K(\log n)^2}{n^2}$, for a constant $K$ large enough (in terms of $d$, $L$ and $\alpha$). Applying Lemma \ref{l.EB} for this choice of $p'$ (and $p=p_c$), and relying on the mean-field bound on the susceptibility $\mathbb E_{p'}[|\mathcal C(0)|]=\chi(p')\lesssim (p_c-p')^{-1}$ (see \cite{AizenmanNewmanTreeGraphInequalities1984}) yields
\begin{equation}
	\mathbb P_{p_c}[0\connect{}\Lambda_n^c]\lesssim \frac{(\log n)^2}{n^2},
\end{equation}
which is a weaker version of the result of \cite{KozmaNachmias}. To remove the $(\log n)^2$ factor above, we consider another notion of correlation length, studied in \cite{DumPan24Perco}.

Let us present the main input in our analysis. We will need some notations. 
If $p\in [0,1]$ and $S\subset \mathbb Z^d$ contains $0$, let
	\begin{equation}
		\varphi_p(S):=p\sum_{\substack{x\in S\\y\notin S\\x\sim y}}\tau^{S}_p(0,x).
	\end{equation}
	Such a quantity was first considered in \cite{DuminilTassionNewProofSharpness2016}. We define another notion of correlation length based on $\varphi_p(S)$ as follows: for every $p<p_c$,
	\begin{equation}\label{e.Lp}
	L(p)\coloneqq \min \{k\geq1: \varphi_p(\Lambda_k)\leq 1/e^2\}.
	\end{equation}
	This natural length scale was first considered in \cite[Section~4]{hutchcroft2022derivation}. The closely related \emph{sharp length} introduced in \cite{PanisTriviality2023} in the context of the Ising model, is also considered below, see \eqref{e.Lp2}. These quantities were studied in several other papers \cite{PanisTriviality2023,DuminilPanis2024newLB,DumPan24WSAW,vEGPS}. By \cite{DuminilTassionNewProofSharpness2016}, one has $\varphi_{p_c}(\Lambda_k)\geq 1$ for every $k\geq 1$. This observation, together with exponential decay of the two-point function below $p_c$ (see \cite{Mensikov1986coincidence,AizenmanBarsky1987sharpnessPerco,DuminilTassionNewProofSharpness2016}) allows to prove that $L(p)$ diverges as $p$ approaches $p_c$.
	
	Recall that the susceptibility $\chi(p)$ is defined, for $p<p_c$, by
	\begin{equation}
		\chi(p):=\sum_{x\in \mathbb Z^d}\tau_p(0,x).
	\end{equation}
	Below, we consider translated half-spaces $\mathbb H_n:=\mathbb H-n\mathbf{e}_1$, for $n\in\Z$ (recall that $\mathbb H:=\{0,1,\ldots\}\times\mathbb Z^{d-1}$). 
	We will base our analysis on the following result from \cite{DumPan24Perco}.
	\begin{Thm}[{\hspace{1pt}\cite[Corollary~1.4]{DumPan24Perco}}]\label{thm:INPUT} Let $d>6$. There exist $L_0\geq 1$ such that the following holds. For every $L\geq L_0$, there exist $c_1,C_1>0$ such that, and for every $p\leq p_c$,
	\begin{equation}\label{eq:input chi(p)}
		c_1(p_c-p)^{-1}\leq \chi(p)\leq C_1(p_c-p)^{-1},
	\end{equation}
	and
	\begin{equation}\label{eq:input L(p)}
		c_1(p_c-p)^{-1/2}\leq L(p)\leq C_1(p_c-p)^{-1/2}.
	\end{equation}
	Moreover, for every $n\geq 0$,
	\begin{equation}\label{eq:input Phi(H_n)}
		\varphi_p(\mathbb H_n)\leq C_1\exp\left(-c_1(p_c-p)^{1/2}n\right).
	\end{equation}
	\end{Thm}
	\begin{Rem} (1) For full disclosure, let us mention that \eqref{eq:input Phi(H_n)} follows from a straightforward combination of \eqref{eq:input L(p)} and \cite[Remark~3.10]{DumPan24Perco}.
	
	(2) The constant $L_0$ appearing in the statement of Theorem \ref{thm: main result} is exactly the one provided by Theorem \ref{thm:INPUT}. This means that we do not rely on additional perturbative arguments in this paper.
	
	(3) In \cite{DumPan24Perco}, the authors prove that $\tau_{p_c}(0,x)\asymp |x|^{2-d}$. Combined with \cite{AizenmanNewmanTreeGraphInequalities1984}, this allows to get \eqref{eq:input chi(p)}. Additionally, let us mention that it plays a key role in the obtention of \eqref{eq:input L(p)} (and hence \eqref{eq:input Phi(H_n)}). We will not, however, use directly this input in our paper.
	
	(4) The bound \eqref{eq:input Phi(H_n)} was also derived in \cite{HutchcroftMichtaSladePercolationTorusPlateau2023} for both nearest-neighbour and spread-out percolation, assuming an up-to-constant estimate on the critical two-point function. However, their proof relies on inputs from \cite{ChatterjeeHanson}, which in turn relies on \cite{KozmaNachmias}. The main goal of this paper is to presents alternatives to these papers. We only rely on \cite{DumPan24Perco}, which explains why we restrict to spread-out percolation.	
	\end{Rem}
	
Let $\theta_n(p)=\mathbb P_{p}[0\connect{}\Lambda_n^c]$. A combination of Theorem \ref{thm:INPUT} and the van den Berg--Kesten inequality (see \eqref{eq:BK} below) gives the existence of $K_1=K_1(d,L)>0$ such that, for $p'=p_c-\frac{K_1}{n^2}$, one has
\begin{equation}
	\theta_n(p')\leq \varphi_{p'}(\Lambda_{(n/2)- L })\theta_{n/2}(p_c) \leq 2d\varphi_{p'}(\mathbb H_{(n/2)- L })\theta_{n/2}(p_c)\leq \frac{1}{8}\theta_{n/2}(p_c).
\end{equation}
Using this input, applying Lemma \ref{l.EB} as in \cite{DewanMuirhead}, and relying on \eqref{eq:input chi(p)} to get $\chi(p')\lesssim (p_c-p)^{-1}\asymp n^2$ yields the following renormalisation inequality: there exists $K_2=K_2(d,L)>0$ such that, for every $n\geq 2$,
\begin{equation}
	\theta_n(p_c)\leq \frac{1}{8}\theta_{n/2}(p_c)+\frac{K_2}{n}\sqrt{\theta_n(p_c)},
\end{equation}
from which it is easy to deduce the upper bound in Theorem \ref{thm: main result}. The lower bound follows easily from Theorem \ref{thm:INPUT} and a differential inequality on $\theta_n(p)$ obtained in \cite{DuminilTassionNewProofSharpness2016}. We provide the details in Section \ref{sec:onearmfull}.

Finally, we briefly describe the proof of Theorem \ref{thm: halfspace one arm}. The upper-bound is proved in two steps. First, we use the same strategy as the full-space case to bound the quantity $\P_{p_c}[0\connect{\mathbb H\:} \H_{-n}]$, except that the full-space susceptibility $\chi(p')$ needs to be replaced by 
\begin{equation}
	\chi^{\mathbb H}(p'):=\sum_{x\in \mathbb H}\tau^{\mathbb H}_{p'}(0,x).
\end{equation} 
As it turns out, it is an easy corollary of \eqref{eq:input Phi(H_n)} that $\chi^{\mathbb H}(p')\lesssim (p_c-p')^{-1/2}$---this explains why the one-arm exponent goes from $2$ to $3$ when restricting to the half-space. 
The second step uses symmetry and an averaging trick to deduce an analogous bound for $\P_{p_c}[0\connect{\mathbb H\:} \Lambda_n^c]$. 
The proof of the lower bound in Theorem \ref{thm: halfspace one arm} uses a second moment method which---surprisingly---does not rely on any pointwise estimates but rather on a combination of Theorem \ref{thm:INPUT} with resummation tricks. The detailed proof is presented in Section \ref{sec:onearmhalf}.

As observed by Hutchcroft \cite{hutchcroft2022derivation}, the combination of the mean-field bound $\chi(p)\lesssim (p_c-p)^{-1}$ and Lemma \ref{l.EB} yields a short proof of mean-field bounds for other observables, like the percolation density or the \emph{magnetisation}. This provides an alternative to the differential inequality approach of \cite{BarskyAizenmanCriticalExponentPercoUnderTriangle1991}. For sake of completeness, we include the derivation of these bounds in the appendix (see Theorem \ref{thm:MF bounds theta}). 
Let us mention that entropic arguments can also be used to derive mean-field lower bounds on volume and magnetisation starting just from a mean-field lower bound on $\theta(p)$, see \cite{DM_Bololean}.

\paragraph{Acknowledgements.} We thank Gady Kozma for suggesting us to look at \cite{DewanMuirhead}, and Hugo Duminil-Copin for useful discussions. We also thank Tom Hutchcroft, Stephen Muirhead, and Akira Sakai for useful comments. DvE, CG and FS acknowledge support from the ERC grant VORTEX 101043450. RP acknowledges support from the SNSF through a Postdoc.Mobility grant.

\section{Preliminaries}\label{sec:preliminaries}
	
We will make repeated use of the BK inequality (see \cite[Section 2.3]{GrimmettPercolation1999}), which we quickly recall now. For every pair of increasing events $A$ and $B$, we have
\begin{equation}\tag{BK}\label{eq:BK}
	\P_p[A\circ B]\leq \P_p[A]\P_p[B],
\end{equation}
where $A\circ B$ denotes the \emph{disjoint occurrence} of $A$ and $B$, i.e.~the event that there exist two disjoint sets $I$ and $J$ of edges such that the configuration restricted to $I$ (resp.~$J$) is sufficient to decide that $A$ (resp.~$B$) occurs.

In this section, we state the main inputs for our proofs. In particular, we derive several consequences of Theorem \ref{thm:INPUT}. We fix $d>6$, and $L_0$ is given by Theorem \ref{thm:INPUT}.
	
\subsection{Bounds on $\varphi_p$}\label{subsec:correlation_lenght}

We begin with a lower bound on $\varphi_{p_c}(\mathbb H_n)$.

\begin{Lem}\label{lem:lower bound phi} Let $d>6$ and $L\geq L_0$. For every $p\in [0,1]$, and every $n\geq 0$,
\begin{equation}
	\varphi_{p}(\mathbb H_n)\geq \frac{1}{2d}\varphi_{p}(\Lambda_n).
\end{equation}
In particular, $\varphi_{p_c}(\mathbb H_n)\geq \tfrac{1}{2d}$.
\end{Lem}
\begin{proof} Observe that by symmetry one has
\begin{equation}
	\varphi_{p}(\Lambda_n)\leq 2d \varphi_{p}(\mathbb H_n).
\end{equation}
The second part of the statement follows from the bound of \cite{DuminilTassionNewProofSharpness2016}: $\varphi_{p_c}(\Lambda_n)\geq 1$.
\end{proof}
For a set $S\subset \mathbb Z^d$, the (inner vertex) boundary of $S$ is $\partial S:=\{x\in S:~\exists y\notin S \text{ s.t. } \vert x-y\vert_1=1\}$ (where $|\cdot|_1$ is the $\ell^1$ norm on $\mathbb R^d$).
For convenience, we sometimes consider a variant of $\varphi_{p}$ defined as follows: for $p\in [0,1]$ and $S\subset \mathbb Z^d$ containing $0$,
	\begin{equation}
		\psi_{p}(S):=\sum_{x\in \partial S} \tau_{p}^{S}(0,x).
	\end{equation}
This quantity is often referred to as the \emph{expected number of pioneers} (of $S$), see for instance \cite{HutchcroftMichtaSladePercolationTorusPlateau2023}. Observe that
\begin{equation}\label{eq:comparison psi phi}
	p\psi_{p}(S)\leq \varphi_{p}(S) \lesssim \psi_{p}(S).
\end{equation}
Indeed, the lower bound is straightforward, while the upper bound follows from the following observation: 
if $x\in S \setminus \p S$ is such that there exists $y\notin S$ with $x\sim y$, then there is $z \in \p S$ such that $x\sim z$ and $z\sim y$. Now, the FKG inequality (see \cite{GrimmettPercolation1999}) implies that
\begin{equation}\label{eq:finite energy arg}
	p\tau_{p}^{S}(0,x)=\mathbb P_p[e=\{x,z\}\textup{ is open}]\mathbb P_p[0\connect{S}x]\leq \tau_{p}^{S}(0,z).
\end{equation}
Since there are at most $C\,  L^d \times L^d$ such pairs $ \{x,y\}$ with $x\in S\setminus \p S, y\notin S, x\sim y$ and which are  sent to the same boundary point $z\in \p S$, this implies the upper bound. 

The following result will be useful, and is an easy consequence of Lemma \ref{lem:lower bound phi}, \eqref{eq:input Phi(H_n)}, and \eqref{eq:comparison psi phi}.
\begin{Coro}\label{coro:bounds on psi}
 Let $d>6$ and $L\geq L_0$. There exist $c,C>0$ such that, for every $p\leq p_c$, and every $n\geq 0$, 
\begin{align}
	 p \psi_p(\Lambda_n) & \leq   \varphi_p(\Lambda_n)  \leq C \exp\left( - c (p_c-p)^{1/2} n\right),\label{e.PsiL} \\
	 p \psi_p(\H_n) & \leq   \varphi_p(\H_n)  \leq C \exp\left( - c (p_c-p)^{1/2} n\right)\label{e.PsiH}\,
\end{align}
and
\begin{equation}\label{e.PsiCrit}
	c\leq \psi_{p_c}(\mathbb H_n)\leq C.
\end{equation}
\end{Coro}
\begin{proof} Recall that by symmetry one has $\varphi_p(\Lambda_n)\leq 2d\varphi_p(\mathbb H_n)$. Combining this observation with \eqref{eq:input Phi(H_n)} gives the upper bounds in \eqref{e.PsiL} and \eqref{e.PsiH}. The lower bounds in these equations follow from \eqref{eq:comparison psi phi}. Finally, \eqref{e.PsiCrit} follows from \eqref{eq:input Phi(H_n)} and Lemma \ref{lem:lower bound phi} thanks to \eqref{eq:comparison psi phi}.
\end{proof}

\subsection{The sharp length}
The quantity $L(p)$ of \eqref{e.Lp} was defined by looking at the first $k\geq 1$ for which $\varphi_p(\Lambda_k)$ drops below $1/e^2$. It can be more practical to consider a slightly different definition in which we rather look at the first $k$ for which there exists a $S\subset \Lambda_k$ such that $\varphi_p(S)$ ``drops''. More precisely, the \emph{sharp length} $\tilde{L}(p)$ is defined, for every $p<p_c$, by
\begin{equation}\label{e.Lp2}
	\tilde L(p)  \coloneqq \min \{k\geq1:~\exists  S \subset \Lambda_k \textup{ containing $0$ such that } \varphi_p(S)\leq 1/e^2\}\,.
\end{equation}
Our next result shows that $\tilde{L}(p)$ and $L(p)$ have the same behaviour as $p$ approaches $p_c$.

\begin{Prop}\label{prop:equivalence of sharp length}
 Let $d>6$ and $L\geq L_0$. There exist $c,C>0$ such that, for every $p<p_c$,
\begin{equation}
	c(p_c-p)^{-1/2}\leq \tilde{L}(p)\leq C(p_c-p)^{-1/2}.
\end{equation}
\end{Prop}
\begin{proof} By \eqref{eq:input L(p)}, it suffices to show that $\tilde{L}(p)\asymp L(p)$. First, observe that (by definition) $\tilde L(p) \leq L(p)$. 
We now prove that $\tilde{L}(p)\gtrsim L(p)$. 

Let $S\subset \Lambda_{\tilde{L}(p)}$ be such that $\varphi_p(S)\leq 1/e^2$. Let $N= (K+1) (\tilde{L}(p)+L)$ for some large enough constant $K\geq 1$ to be chosen later. We will show that $L(p)\leq N$, which will conclude the proof. 
Let $x\in \partial \Lambda_N$ and observe that
\begin{equation}
\{0\connect{\Lambda_N\:} x\} \subset \bigcup_{\substack{u_1\in S\\ v_1\notin S\\ u_1\sim v_1}} \{0\connect{S\:} u_1\} \circ \{u_1v_1 \text{ is open}\} \circ \{v_1 \connect{\Lambda_N\:} x\}.
\end{equation}
 Using \eqref{eq:comparison psi phi} and \eqref{eq:BK}, we obtain 
\begin{equation}
\varphi_p(\Lambda_N)\lesssim  \psi_p(\Lambda_N) \leq \sum_{x\in\partial\Lambda_N}  \sum_{\substack{u_1\in S\\ v_1\notin S \\ u_1\sim v_1}} \tau_p^S(0,u_1)\,p\,  \tau_p^{\Lambda_N}(v_1,x).
\end{equation}
By iterating this expansion $K$ times, we get (with the convention $v_0=0$)
\begin{equation}\label{eq:iterated-phi}
	\varphi_p(\Lambda_N)\lesssim  \sum_{x\in\partial\Lambda_N}  
	\sum_{\substack{u_1\in S\\ v_1\notin S \\ u_1\sim v_1}} \tau_p^S(v_0,u_1)\,p
	\ldots\sum_{\substack{u_{K}\in (S+v_{K-1})\\ v_{K}\notin (S+v_{K-1}) \\ u_K\sim v_K}} \tau_p^S(v_{K-1},u_K)\,p\, \tau_p^{\Lambda_N}(v_K,x).\end{equation}
Observe that by construction, any $v_K$ as above must lie in $\Lambda_N$, and 
\begin{equation}\label{eq:bound_vK}
\sum_{x\in  \p \Lambda_N} \tau_p^{\Lambda_N}(v_K,x)=\psi_{p}(\Lambda_N-v_K)\leq \psi_{p_c}(\Lambda_N-v_K)   
\leq 2dp_c^{-1} \max_{n\geq 0}\varphi_{p_c}(\mathbb H_n)\lesssim 1,
\end{equation}
where in the second inequality we used \eqref{eq:comparison psi phi}, and in the last one \eqref{eq:input Phi(H_n)}. Plugging \eqref{eq:bound_vK} in \eqref{eq:iterated-phi} and using translation invariance, we obtain
\begin{equation}\label{}
\varphi_p(\Lambda_N)  \lesssim  \varphi_p(S)^K \leq e^{-2K}.
\end{equation}
If $K$ is large enough, we deduce that $\varphi_p(\Lambda_N)\leq 1/e^2$, so that $L(p)\leq N\lesssim \tilde{L}(p)$. This concludes the proof.
\end{proof}

\subsection{Bounds on the half-space susceptibility}

In this subsection, we prove a useful bound in our proof of Theorem \ref{thm: halfspace one arm}. We begin with a definition. If $p\in[0,1]$, the half-space susceptibility is defined by
\begin{equation}
	\chi^{\mathbb H}(p):=\sum_{x\in \mathbb H}\tau_{p}^{\mathbb H}(0,x).
\end{equation}
The next result should be compared to \eqref{eq:input chi(p)}.

\begin{Prop}\label{prop:half-space_susceptibility}
Let $d>6$ and $L\geq L_0$. There exist $c_2,C_2>0$ such that, for every $\tfrac{p_c}{2}\leq p<p_c$,
\begin{equation}\label{eq:half-space_susceptibility}
	c_2(p_c-p)^{-1/2}\leq\chi^{\mathbb H}(p)\leq C_2(p_c-p)^{-1/2}.
\end{equation}
\end{Prop}

\begin{proof} By translation invariance, one has
\begin{equation}\label{eq:phalfspace chi 1}
	\chi^{\mathbb H}(p)= \sum_{k\geq 0}\psi_p(\mathbb H_k).
\end{equation}
Combining \eqref{eq:comparison psi phi} and Lemma \ref{lem:lower bound phi}, we find that
\begin{equation}
	\psi_p(\mathbb H_k)\gtrsim \varphi_{p}(\Lambda_k).
\end{equation}
In particular, when $k<L(p)$, one has $\psi_p(\mathbb H_k)\gtrsim 1$. Plugging this observation in \eqref{eq:phalfspace chi 1} yields
\begin{equation}\label{eq:phalfspace chi 2}
	\chi^{\mathbb H}(p)\gtrsim L(p) \gtrsim (p_c-p)^{-1/2},
\end{equation}
where the second inequality follows from \eqref{eq:input L(p)}. Using Corollary \ref{coro:bounds on psi}, we find $c,C>0$ such that, for every $p_c/2\leq p<p_c$, and every $k\geq 0$, 
\begin{equation}\label{eq:phalfspace chi 3}
	\psi_{p}(\mathbb H_k)\leq C(p_c/2)^{-1}\exp\left(-c(p_c-p)^{-1/2}k\right).
\end{equation}
Plugging \eqref{eq:phalfspace chi 3} in \eqref{eq:phalfspace chi 1} gives
\begin{equation}
	\chi^{\mathbb H}(p)\lesssim \sum_{k\geq 0}\exp\left(-c(p_c-p)^{-1/2}k\right)\lesssim (p_c-p)^{-1/2}.
	\end{equation}
	This concludes the proof.
\end{proof}

\section{One-arm probability in the full-space}\label{sec:onearmfull}

In this section we prove Theorem~\ref{thm: main result}. We fix $d>6$ and $L\geq L_0$, where $L_0$ is given by Theorem \ref{thm:INPUT}. We begin with a proof of the (more interesting) upper bound.

\begin{proof}[Proof of the upper bound in \eqref{eq:main_bulk}] Assume first that $n\geq 2L$.
	We let 
	\begin{equation}
	\theta_n(p)=\mathbb{P}_{p}[0\connect{}\Lambda_n^c].
	\end{equation} 
	By applying the entropic bound of Lemma \ref{l.EB} with $\Lambda=\mathbb Z^d$ and $\mathcal A:=\{0\connect{}\Lambda_n^c\}$, we obtain that for any $p_c/2\leq p'<p_c$, 
	\begin{align}\label{e.EB}
		\theta_n(p_c) - \theta_n(p') \leq C_0 (p_c-p')  \sqrt{ \theta_n(p_c)  \chi(p') }\,,
	\end{align}
	where we used the fact that $\theta_n(p') \leq \theta_n(p_c)$. 
	Let us now choose $p'$, sufficiently far in the near-critical window, i.e. 
	\begin{align*}\label{}
		p' := p_c - K n^{-2}\,,
	\end{align*}
	where $K$ is a constant chosen large enough so that, by Corollary \ref{coro:bounds on psi}, one has $\varphi_{p'}(\Lambda_{(n/2)-L})\leq 1/8$. Note that, by decomposing an open path from $0$ to $\Lambda_n^c$ according to the first edge exiting $\Lambda_{(n/2)-L}$, we have
	\begin{align}
	\begin{split}
		\{0\connect{}\Lambda_n^c\} &\subset \bigcup_{\substack{x\in  \Lambda_{(n/2)-L}\\y \notin \Lambda_{(n/2)-L}\\x \sim y}} \{0\connect{\Lambda_{(n/2)-L}\:} x\} \circ \{xy \textup{ open}\}\circ \{y\connect{}\Lambda_n^c\} \\
		&\subset \bigcup_{\substack{x\in  \Lambda_{(n/2)-L}\\y \notin \Lambda_{(n/2)-L}\\x \sim y}}\{0\connect{\Lambda_{(n/2)-L}\:} x\} \circ \{xy \textup{ open}\}\circ  \{y\connect{}(\Lambda_{n/2} + y)^c\},
	\end{split}
	\end{align}
	so by \eqref{eq:BK} and translation invariance, we obtain
	\begin{align}
	\begin{split}
		\theta_n(p') &\leq \sum_{\substack{x\in  \Lambda_{(n/2)-L}\\y \notin \Lambda_{(n/2)-L}\\x \sim y}} \mathbb{P}_{p_c}[0\connect{\Lambda_{(n/2)-L}\:} x] p' \mathbb{P}_{p'}[y\connect{}(\Lambda_{n/2}+y)^c]\\ 
		& = \varphi_{p'}(\Lambda_{(n/2)-L})  \theta_{n/2}(p') \leq \frac 1 8  \theta_{n/2}(p_c)\,,
	\end{split}
	\end{align}
	where in the last line we used that $\theta_{n/2}(p')\leq \theta_{n/2}(p_c)$.
	Plugging this into~\eqref{e.EB}, and recalling that $\chi(p')\leq C_1K^{-1} n^{2}$ by Theorem~\ref{thm:INPUT}, gives, for every $n\geq 2L$
	\begin{equation}\label{eq:UB_bulk_renormalization}
		\theta_n(p_c)  \leq  \frac{1}{8}\theta_{n/2}(p_c) + C_0 (C_1K)^{1/2}  n^{-1}\sqrt{ \theta_n(p_c)   }\,.
	\end{equation}
	Let $k_0\geq 1$ be such that $2^{k_0}\geq 2L>2^{k_0-1}$.
	We claim that the inequality above implies that, for every $k\geq k_0$,
	\begin{equation}\label{eq:UB_bulk_induction}
		\theta_{2^{k}}(p_c)   \leq \frac {A} {(2^{k})^2},
	\end{equation}
	where $A:=\max\{4C_0^2C_1K,(2^{k_0})^2\}$. This in turn readily implies the desired bound. Indeed, for every $n\geq 1$, if $n\leq 2^{k_0}$, then $\theta_n(p_c)\leq 1\leq An^{-2}$ by choice of $A$. Otherwise, there is $k\geq k_0$ such that $2^k\leq n < 2^{k+1}$, so that $\theta_n(p_c)\leq \theta_{2^{k}}(p_c)   \leq A 2^{-2k} \leq 4An^{-2}$, as desired.
	
	Let us now prove that \eqref{eq:UB_bulk_induction} holds by induction. First, since $\theta_{2^{k_0}}(p_c)\leq1 \leq A/(2^{k_0})^2$, \eqref{eq:UB_bulk_induction} is trivially true for $k=k_0$. Assuming it is true for a given $k\geq k_0$ and applying \eqref{eq:UB_bulk_renormalization} for $n=2^{k+1}$, we obtain
		\begin{equation}
		\theta_n(p_c)  \leq  \frac{A}{2n^2} +  \frac {\sqrt{A}}{2n}  \sqrt{ \theta_n(p_c)   } \implies \left( \sqrt{\theta_n(p_c)}  -  \frac{\sqrt{A}}{4n} \right)^2\leq \frac{9A}{16n^2} \implies \theta_n(p_c)\leq \frac{A}{n^2}\,.
		\end{equation}
	This concludes the proof.
\end{proof}

 We now prove the lower bound in Theorem \ref{thm: main result}. Our argument only uses the lower bound $\tilde{L}(p)\gtrsim(p_c-p)^{-1/2}$. We note that an alternative---and technically simpler---proof could be obtained via a second-moment method, relying on pointwise estimates on the critical two-point function (as first carried out in \cite{sakai2004mean}).

\begin{proof}[Proof of the lower bound in \eqref{eq:main_bulk}] 
	Again, we let $\theta_n(p)=\mathbb P_p[0\connect{}\Lambda_n^c]$.
	We will use the following differential inequality, which is proved in \cite[Lemma~1.4]{DuminilTassionNewProofSharpness2016}: there exists a constant $c=c(d,L)>0$ such that for every $p\in [p_c/2,1]$ and $n\geq1$, one has 
	\begin{equation}\label{eq:diff_ineq_theta}
		\frac{\mathrm{d}}{\mathrm{d}p}\theta_n(p)\geq c\inf_{0\in S\subset \Lambda_n}\varphi_p(S)(1-\theta_n(p))\,.
	\end{equation}
	Recall from Proposition~\ref{prop:equivalence of sharp length} that $\tilde L(p) \asymp (p_c-p)^{-1/2}$.  By the definition of $\tilde{L}(p)$, this implies that there exists $c_3>0$ such that for every $n$ large enough and every $p\in (p_c- \frac{c_3}{n^2}, p_c]$, one has 
	\begin{equation}\label{eq:lb phi(S)}
		\varphi_p(S) \geq 1/e^{2}\,, ~~~ \forall \,   S \subset \Lambda_n\, \textup{ containing $0$}.
	\end{equation}
	Moreover, by the upper bound on $\theta_n(p_c)$ we just proved, the sequence $(\theta_n(p_c))_{n\geq 1}$ tends to $0$, and therefore for every $n$ large enough and $p\leq p_c$, one has $1-\theta_n(p)\geq \tfrac{1}{2}$. Plugging this observation and \eqref{eq:lb phi(S)} into \eqref{eq:diff_ineq_theta} gives $c_2>0$ such that, for every $n$ large enough, and every $p\in (p_c-\tfrac{c_3}{n^2}, p_c]$, one has
	\begin{equation}
			\frac{\mathrm{d}}{\mathrm{d}p}\theta_n(p)\geq c_4\,. 
	\end{equation}
	Integrating this over $p\in (p_c-\tfrac{c_3}{n^2}, p_c]$, we obtain, for every $n$ large enough,
	\begin{equation}\label{}
		\theta_n(p_c) \geq \frac{c_3c_4}{n^2}\,.
	\end{equation}
	This concludes the proof.
\end{proof}

\section{One-arm probability in the half-space}\label{sec:onearmhalf}

We now turn to the proof of Theorem \ref{thm: halfspace one arm}.  We fix $d>6$ and $L\geq L_0$, where $L_0$ is given by Theorem \ref{thm:INPUT}.

\subsection{Proof of the lower bound}

Our goal in this section is to prove the lower bound in \eqref{eq:main_half-space}.
The proof is based on a second moment method. This would be a relatively easy task if we knew precise estimates on the two point function restricted to $\H$. However, such estimates are difficult to obtain. In fact, they were first proved \cite{ChatterjeeHanson} (see also \cite{ChatterjeeHansonSosoe2023subcritical}), through a delicate analysis which is based, in particular, on the results of \cite{KozmaNachmias}.
Instead of relying on such pointwise bounds on the restricted two-point function, we provide a softer proof which only uses the integrated estimate provided by \eqref{e.PsiCrit} as an input. 

\medskip

\noindent
\begin{proof}[Proof of the lower bound in \eqref{eq:main_half-space}] Let $n\geq 1$.
	Consider the $n$-slab defined as
	\begin{align}\label{}
		\mathbb{S}_n: = \H \setminus \H_{-(n+1)}.
	\end{align}
	We will apply the second moment method to the following random variable 
	\begin{align}\label{}
		\mathcal N_n:= \# \{ x \in  \p \H_{-n}:  0\connect{\mathbb S_n\:} x  \}.
	\end{align}
	
	\begin{Lem}\label{l.N2}
		For every $n\geq1$,
		\begin{align}\label{}
			\Eb{\mathcal N_n^2} \lesssim n.
		\end{align}
	\end{Lem}
	
	\begin{Lem}\label{l.N1}
		For every $n\geq1$,
		\begin{align}\label{}
			\Eb{\mathcal N_n} \gtrsim \frac{1}{n}.
		\end{align}
	\end{Lem}
	
	Lemmas~\ref{l.N2} and \ref{l.N1} readily imply the desired lower bound: by the Cauchy--Schwarz inequality,
	\begin{align}\label{}
		\P_{p_c}[0\connect{\mathbb H\:} \Lambda_{n-1}^c] \geq \P_{p_c}[0\connect{\mathbb{S}_n\:} \partial \H_{-n}] 
		= \P_{p_c}[ \mathcal N_n \geq 1] 
		\geq \frac {\Eb{\mathcal N_n}^2} {\Eb{\mathcal N_n^2}}  \gtrsim \frac{1}{n^3}.
	\end{align}
\end{proof}

\begin{proof}[Proof of Lemma \textup{\ref{l.N2}}]
	By definition, 
	\begin{equation}
		\Eb{\mathcal N_n^2}=\sum_{x,y\in \partial \H_{-n}}  \P_{p_c}[0\connect{\mathbb{S}_n\:}x,~0\connect{\mathbb{S}_n\:}y].
	\end{equation}
	By considering the last intersection between a pair of open paths in $\mathbb{S}_n$ from $0$ to $x$ and from $0$ to $y$, we deduce that $\{0\connect{\mathbb{S}_n\:}x\}\cap \{0\connect{\mathbb{S}_n\:}y \} \subset \bigcup_{u\in \mathbb{S}_n} \{0\connect{\mathbb{S}_n\:}u\} \circ \{u\connect{\mathbb{S}_n\:}x\} \circ \{u\connect{\mathbb{S}_n\:}y\}$. Applying \eqref{eq:BK}, we obtain
	\begin{align}\label{eq:proof lb half space 1}
	\begin{split}
		\Eb{\mathcal N_n^2}&\leq \sum_{k=0}^n \sum_{u\in \partial \H_{-k}} \sum_{x,y\in \partial\H_{-n}} \tau_{p_c}^{\mathbb{S}_n}(0,u) \tau_{p_c}^{\mathbb{S}_n}(u,x) \tau_{p_c}^{\mathbb{S}_n}(u,y) \\
		&\leq \sum_{k=0}^n \sum_{u\in \partial \H_{-k}} \sum_{x,y\in \partial\H_{-n}} \tau_{p_c}^{\H}(0,u) \tau_{p_c}^{(\mathbb H_{-n})^c}(u,x) \tau_{p_c}^{(\H_{-n})^c}(u,y) \\
		&=\sum_{k=0}^n \psi_{p_c}(\H_k) \psi_{p_c}(\H_{n-k})^2 \lesssim n,
		\end{split}
	\end{align}
	where in the second line we used inclusion, and in the third line we used translation invariance and \eqref{e.PsiCrit}.
\end{proof}

\begin{proof}[Proof of Lemma \textup{\ref{l.N1}}] Let $n\geq 1$. On the one hand, combining \eqref{e.PsiCrit} and translation invariance gives 
	\begin{equation}\label{eq:H_lowerbound}
	 \sum_{x\in \partial \H_{-n}} \tau_{p_c}^{\H}(0,x) = \psi_{p_c}(\H_n) \gtrsim 1.
	\end{equation}
	On the other hand, by considering the step immediately before the first intersection of an open path from $0$ to $x$ with $\partial\H_{-n}$, we deduce that 
	\begin{equation}
	\{0\connect{\H\:}x\}\subset \bigcup_{1\leq k\leq L} \bigcup_{u\in \partial \H_{-n+k}} \{0\connect{\mathbb{S}_n\:}u\} \circ \{u\connect{\H\:}x\}.
	\end{equation}
	 Applying \eqref{eq:BK}, we obtain
	\begin{align}
		\begin{split}\label{eq:H-to-S}
			\sum_{x\in \partial \H_{-n}} \tau_{p_c}^{\H}(0,x) &\leq \sum_{k=1}^L \sum_{x\in \partial \H_{-n}}  \sum_{u\in \partial \H_{-n+k}} \tau_{p_c}^{\mathbb{S}_n}(0,u)\tau_{p_c}^{\H}(u,x) \\
			&= \sum_{k=1}^L \Bigg( \sum_{u\in \partial \H_{-n+k}} \tau_{p_c}^{\mathbb{S}_n}(0,u) \Bigg) \Bigg( \sum_{y\in\partial\H_{k}} \tau_{p_c}^{\H_{n}}(0,y) \Bigg),
		\end{split}
	\end{align}
	where in the last line we used translation invariance to obtain, for each $u\in \partial \mathbb H_{-n+k}$,
	\begin{equation}
		\sum_{x\in \partial \H_{-n}}\tau^{\mathbb H}_{p_c}(u,x)=\sum_{y\in\partial\H_{k}} \tau_{p_c}^{\H_{n}}(0,y).
	\end{equation} 
	Fix $1\leq k\leq L$ and $y\in \partial \H_k$. By decomposing an open path from $0$ to $y$ according to the left-most vertex it visits, we obtain $\{0\connect{\H_n\:}y\}\subset \bigcup_{k\leq i\leq n,\, v\in \partial\H_i} \{0\connect{\H_i\:}v\} \circ \{v\connect{\H_i\:}y\}$. Applying \eqref{eq:BK}, we obtain
	\begin{align}
		\begin{split}\label{eq:H_upperbound}
			\sum_{y\in\partial\H_{k}} \tau_{p_c}^{\H_n}(0,y) &\leq \sum_{i=k}^{n} \sum_{v\in \partial \H_i} \sum_{y\in\partial\H_{k}}  \tau_{p_c}^{\H_i}(0,v) \tau_{p_c}^{\H_i}(v,y) \\
			&= \sum_{i=k}^{n} \psi_{p_c}(\H_i)\psi_{p_c}(\H_{i-k}) \lesssim n,
		\end{split}
	\end{align}
	where in the second line we used translation invariance and \eqref{e.PsiCrit}. 
	Combining \eqref{eq:H_lowerbound}, \eqref{eq:H-to-S} and \eqref{eq:H_upperbound}, we obtain
	\begin{equation}\label{eq:}
		\sum_{k=1}^L \sum_{u\in \partial \H_{-n+k}} \tau_{p_c}^{\mathbb{S}_n}(0,u) \gtrsim n^{-1}.
	\end{equation}
	Now, by a simple comparison argument (similar to that of \eqref{eq:finite energy arg}), one can easily show that
	\begin{equation}
		\Eb{\mathcal N_n}=\sum_{x\in \partial \H_{-n}} \tau_{p_c}^{\mathbb{S}_n}(0,x) \geq \frac{1}{L}\sum_{k=1}^L \sum_{u\in \partial \H_{-n+k}} p_c\tau_{p_c}^{\mathbb{S}_n}(0,u),
	\end{equation}
	which combined with \eqref{eq:} concludes the proof.
\end{proof}

\subsection{Proof of the upper bound}

The proof of the upper-bound in~\eqref{eq:main_half-space} is divided into two steps:
\bnum
\item We first derive an upper bound on ``point to half-space probability'' $\P_{p_c}[0\connect{\mathbb H} \H_{-n}]$ using the same strategy as in the full-space setting.
\item Then, combining a suitable averaging trick with the BK inequality, we promote this upper bound to an upper bound on the larger quantity $\P_{p_c}[0\connect{\mathbb H\:} \Lambda_{n}^c]$. 
\enum

\noindent
\textbf{Step 1. Point to half-space in $\H$.} We will prove that, for every $n\geq 1$,
\begin{equation}\label{eq:pt-to-half-space_ub}
	\P_{p_c}[0\connect{\mathbb H\:} \H_{-n}] \lesssim n^{-3}.
\end{equation}
The proof is very similar to the one of the upper bound in Theorem \ref{thm: main result}, but there are two key differences. First, we use a \emph{last exit} decomposition instead of \emph{first exit} in order to compare scales $n$ and $n/2$ in the near-critical regime. Second, we rely on the bound \eqref{eq:half-space_susceptibility} on the half-space susceptibility, which is smaller than that of the full-space susceptibility \eqref{eq:input chi(p)}. This last point is responsible for the exponent changing from $2$ to $3$. 

Let $\theta^\H_n(p):= \mathbf{P}_{p}[0\connect{\H\:}  \H_{-n}]$ and $p':=p_c-Kn^{-2}$, for some constant $K$ to be chosen large enough below.
By decomposing an open path from $0$ to $\mathbb H_{-n}$ according to the point immediately after the \emph{last} intersection with $\H_{-\frac{n}{2}}^c$ and applying \eqref{eq:BK} we obtain
\begin{equation}\label{eq:half-space_Simon-Lieb}
	\theta_n^{\H}(p') \leq \sum_{k=0}^{L-1} \sum_{u\in \partial \H_{-\frac{n}{2} - k}} \tau^{\H}_{p'}(0,u) \mathbb{P}_{p'}[u \connect{\H_{-\frac{n}{2}}\:} \H_{-n}] \,.
\end{equation}
By translation invariance and a simple comparison argument (similar to that of \eqref{eq:finite energy arg}), for every $u$ as above we have 
\begin{equation}	
\mathbb{P}_{p'}[u \connect{\H_{-\frac{n}{2}}\:} \partial \H_{-n}]\leq \mathbb{P}_{p_c}[u \connect{\H_{-\frac{n}{2}}\:} \partial \H_{-n}] \leq \frac{1}{p_c} \theta^{\H}_{n/2}(p_c).
\end{equation}
Also by translation invariance, for every $k$ as above we have 
\begin{equation}
\sum_{u\in\partial \H_{-\frac{n}{2} - k}} \tau^{\H}_{p'}(0,u) = \psi_{p'}(\H_{\frac{n}{2} + k}).
\end{equation} 
Finally, by \eqref{e.PsiH}, we can choose $K$ large enough so that $\frac{1}{p_c} \sum_{k=0}^{L-1} \psi_{p'}(\H_{\frac{n}{2} + k})\leq 1/16$. 
Combining these observations with \eqref{eq:half-space_Simon-Lieb}, we obtain
\begin{equation}\label{eq:half-space_Simon-Lieb2}
	\theta_n^{\H}(p') \leq \frac{1}{16} \theta^{\H}_{\frac{n}{2}}(p_c) \,.
\end{equation}
Applying Lemma~\ref{l.EB} to $\Lambda=\mathbb H$ and $\mathcal A=\{0\connect{\mathbb S_n\:}\partial \H_{-n}\}$, and recalling that $\chi^\H(p')\leq C_2(p_c-p')^{-1/2}$ by \eqref{eq:half-space_susceptibility}, we obtain
\begin{align}\label{}
\begin{split}
	\theta_n^{\H}(p_c) &\leq \theta_n^\H(p')+C_0  (p_c - p')  \sqrt{\theta_n^\H(p_c) \chi^\H(p')} \\ 
	&\leq \frac{1}{16} \theta^{\H}_{\frac{n}{2}}(p_c) + C_0 (C_2K)^{1/2} n^{-3/2}  \sqrt{\theta_n^\H(p_c)} \,.
\end{split}
\end{align}
Arguing exactly as in full-space case, the inequality above easily implies \eqref{eq:pt-to-half-space_ub} by induction. \qed

\medskip

\noindent
\textbf{Step 2. Point to box in $\H$.} We first assume that $n\geq 6L$.
First, we notice that the event $\{ 0\connect{\mathbb H\:}\Lambda_{n-1}^c \}$ can be achieved either by connecting $0$ to $\H_{-n}$ within $\H$ (which is already taken care of by our first step), or by connecting $0$ to one of the $2(d-1)$ sides of $\Lambda_n$ in the directions of the basis vectors $\mathbf{e}_2,\ldots, \mathbf{e}_d$. Without loss of generality, let us focus on the $\mathbf{e}_2$ direction. For every $k\geq 1$, let $S_k$ be the {\em slice} defined by
\begin{align}\label{}
	S_k :=  \{ x=(x_1,x_2,\ldots, x_d)\in \Z^d \text{ s.t. } 0\leq x_1 < n+L \text{ and } x_2 =k\} 
\end{align}
By symmetry and union bound and taking into account the $L$-range of our spread-out model, we  have
\begin{align}\label{e.faces}
	\P_{p_c}[0\connect{\mathbb H\:}\Lambda_{n-1}^c ] & \leq  \P_{p_c}[0\connect{\mathbb H\:} \H_{-n}] + 2(d-1) \sum_{j=0}^{L-1} \P_{p_c}[0\connect{\H \setminus \H_{-n}\:} S_{n+j}]\,.
\end{align} 

\begin{figure}[htb]
	\begin{center}
		\includegraphics[width=0.65\textwidth]{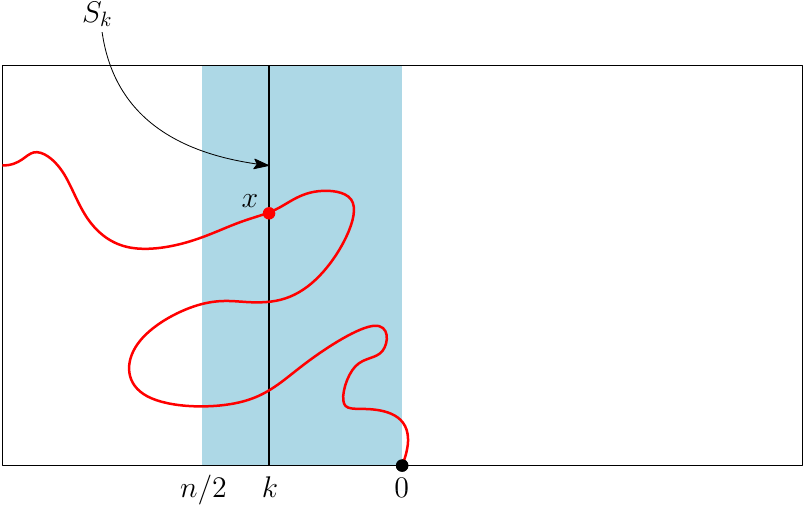}
	\end{center}
	\caption{An illustration of the decomposition used in \eqref{eq:final step up hs}}
	\label{f.slice}
\end{figure}

For each $0\leq j < L$ and each fixed $\frac n 4 \leq  k \leq \frac n 3$, on the event $\{ 0 \connect{\H\:} S_{n+j} \}$, there exists $k'$ with $|k'-k|\leq L/2$ and $x\in S_{k'}$ such that 
$x$ is connected to $S_{n+j}$ inside $\bigcup_{\ell=k'}^{n+j} S_\ell$ and $x$ is connected to 0 inside $\H$. ($x$ can be found as the last visit of the set $\bigcup_{|k'-k|\leq L/2}S_{k'}$ along some exploration procedure of the cluster of $0$ inside $\H$). See Figure \ref{f.slice} for an illustration.  
By union bound and \eqref{eq:BK}, we thus find, for every $\frac n 4 \leq k \leq \frac n 3$ that 
\begin{align}\label{eq:final step up hs}
\begin{split}
	\P_{p_c}[0\connect{\H \setminus \H_{-n}} S_{n+j}]
	& \leq   \sum_{k':|k'-k| \leq L/2} \sum_{x\in S_{k'}}  
	\P_{p_c}[0\connect{\H\:} x] 
	\P_{p_c}[x \connect{\bigcup_{\ell=k'}^{n+j} S_\ell\:} S_{n+j}] \\
	& \leq \sum_{k':|k'-k| \leq L/2}
	\sum_{x\in S_{k'}}  
	\P_{p_c}[0\connect{\H\:} x] 
	\P_{p_c}[0 \connect{\H\:} \H_{-n/2}] \\
	& \lesssim 
	\frac 1 {n^3} \sum_{k':|k'-k| \leq L/2} \sum_{x\in S_{k'}}  
	\P_{p_c}[0\connect{\H\:} x] \,,
	\end{split}
\end{align}
where the second line uses rotational invariance and the fact that the slices $S_{k'}$ are at distance at least $n/2$ from $S_{n+j}$ when $\frac n 4 \leq k \leq \frac n 3$ and $n\ge 6L$, and the last line uses \eqref{eq:pt-to-half-space_ub}.

It is natural to expect that for each such $k$, one has $\sum_{|k'-k|\leq L/2}\sum_{x\in S_{k'}}  
\P_{p_c}[0\connect{\H\:} x]\lesssim 1$. We do not know how to show this with our assumptions. To overcome this, we average \eqref{eq:final step up hs} over $k\in \{ \lfloor \frac n 4 \rfloor  , \ldots, \lfloor \frac n 3 \rfloor \}$. This gives 

\begin{align}\label{eq:last equation}
\begin{split}
\P_{p_c}[0\connect{\H \setminus \H_{-n}\:} S_{n+j}]	
	 & \leq  
\frac{L}{\lfloor \frac n 3 \rfloor - \lfloor \frac n 4 \rfloor+1}  \frac 1 {n^3}\sum_{k'=\lfloor \frac n 4 \rfloor - \frac L 2}^{\lfloor \frac n 3 \rfloor + \frac L 2 }
	\sum_{x\in S_{k'}}  
	\P_{p_c}[0\connect{\H\:} x]   
	\\
	& \lesssim  \frac{1}{n^4}  \sum_{x\in \H \setminus \H_{-(n+L)}}  \P_{p_c}[0\connect{\H\:} x] \\
	& \lesssim  \frac{1}{n^4}\sum_{j=0}^{n+L-1} \psi_{p_c}(\H_j)  \\
	& \lesssim \frac{1}{n^3}\,.
	\end{split}
\end{align}
By plugging this into \eqref{e.faces}, this concludes the proof of the upper bound in \eqref{eq:main_half-space} when $n\geq 6L$. The remaining values of $n$ are treated by using that $\mathbb P_{p_c}[0\connect{\mathbb H\:}\Lambda_n^c]\leq 1$, and adjusting the value of the constant in the last line of \eqref{eq:last equation}.
\qed

\appendix

\section{Percolation density, volume and magnetisation}\label{sec:percodensity}

In this Appendix we present a short proof, due to Hutchcroft \cite{hutchcroft2022derivation}, of mean-field upper bounds for the percolation density, volume and magnetisation.  These bounds were first obtained by Barsky and Aizenman \cite{BarskyAizenmanCriticalExponentPercoUnderTriangle1991} through a delicate use of differential inequalities and under the triangle condition \eqref{eq:triangle condition}. 
The proof below relies on the entropic bound of Lemma~\ref{l.EB} and only uses as an input the mean-field upper bound on the susceptibility, namely $\chi(p)\lesssim (p_c-p)^{-1}$ for every $p<p_c$. 
Let us mention that the main goal of \cite{hutchcroft2022derivation} is to derive upper bounds on $\chi(p)$ from upper bounds on the triangle diagram in a quantitatively better way than in \cite{BarskyAizenmanCriticalExponentPercoUnderTriangle1991}. In our case, such a bound on the susceptibility is directly provided by Theorem~{\ref{thm:INPUT}.

We start by defining the magnetisation $M(p,h)$, introduced in \cite{AizenmanBarsky1987sharpnessPerco} by analogy with the corresponding notion for the Ising model. Consider the \emph{ghost vertex} $\fg$ and the augment percolation on $\mathbb Z^d\cup \{\fg\}$ defined as follows: edges in $E=E_L$ are, as before, independently open with probability $p$, and additionally for every $x\in \mathbb Z^d$, the edge $x\fg$ is open independently of all the other edges with probability $1-e^{-h}$. Then, calling $\mathbb P_{p,h}$ the percolation measure on this enlarged graph, we define
\begin{equation}
	M(p,h)=\mathbb P_{p,h}[0\connect{}\fg].
\end{equation}
Notice that this quantity is directly related to the distribution of the volume of the cluster of the origin in $\Z^d$. Indeed, for $p\leq p_c$ and $h\geq 0$:
\begin{equation}
	M(p,h):=\sum_{n\geq 1}\mathbb P_p[|\mathcal C(0)|=n](1-e^{-nh}).
\end{equation}

It is well known that the mean-field lower bounds $\theta(p)\geq c(p-p_c)$, for every $p>p_c$, and $M(p_c,h)\geq ch^{1/2}$ for every $h>0$, hold in all dimensions $d\geq 2$.
They were derived through a differential inequality approach in \cite{AizenmanBarsky1987sharpnessPerco} (see also \cite{DuminilTassionNewProofSharpness2016} for an alternative proof). 
The following result shows that the corresponding upper bounds hold in dimensions $d>6$. 

\begin{Thm}\label{thm:MF bounds theta} Let $d>6$. Let $L_0$ be given by Theorem~\textup{{\ref{thm:INPUT}}}, and let $L\geq L_0$.  There exists $A\in(0,\infty)$ such that, for every $p>p_c$,
	\begin{equation}\label{eq:MF bound theta}
		\theta(p)\leq A(p-p_c),
	\end{equation}
	and, for every $n\geq1$,
	\begin{equation}\label{eq:MF volume}
		\mathbb{P}_{p_c}[|\mathcal{C}(0)|\geq n]\leq An^{-1/2}\,.
	\end{equation}
	In particular, for every $h\geq 0$,
	\begin{equation}\label{eq:MF bound M}
		M(p_c,h)\leq \sqrt{\pi}Ah^{1/2}.
	\end{equation}
\end{Thm}

\begin{proof}[Proof of \eqref{eq:MF bound theta}] 
	Let $p>p_c$ and consider the symmetrical subcritical parameter $p':=p_c - (p-p_c)$. Applying Lemma~\ref{lem:lower bound phi} for $\Lambda=\Z^d$ and $\mathcal{A}=\{0\connect{}\infty\}$, we obtain
	\begin{equation*}
		\theta(p) = \theta(p)-\theta(p')\leq C_0(p-p')\sqrt{\theta(p) \chi(p')} \leq 2C C_0^{1/2} (p-p_c)^{1/2} \sqrt{\theta(p)},
	\end{equation*}
	where in the last inequality we used that $\chi(p')\leq C_1 (p-p_c)^{-1}$ by Theorem~\ref{thm:INPUT}. 
	The inequality above readily implies that $\theta(p)\leq 4C_0^2 C_1 (p-p_c)$, which concludes the proof.
\end{proof}

\begin{proof}[Proof of \eqref{eq:MF volume}]
	Applying Lemma~\ref{lem:lower bound phi} for
$p':=p_c-n^{-1/2}$,  $\Lambda=\Z^d$, and $\mathcal{A}=\{|\mathcal{C}(0)|\geq n\}$, we obtain
	\begin{align}\label{eq:vol_tail1}
		\begin{split}
			\mathbb{P}_{p_c}[|\mathcal{C}(0)|\geq n] - \mathbb{P}_{p'}[|\mathcal{C}(0)|\geq n] 
			&\leq C_0n^{-1/2}\sqrt{\mathbb{P}_{p_c}[|\mathcal{C}(0)|\geq n] \chi(p')} \\
			&\leq C_0 C_1^{1/2} n^{-1/4} \sqrt{\mathbb{P}_{p_c}[|\mathcal{C}(0)|\geq n]}\,,
		\end{split}
	\end{align}
	where in the second line we used that $\chi(p')\leq C_1 n^{1/2}$ by Theorem~\ref{thm:INPUT}. By Markov's inequality, we have
	\begin{equation}\label{eq:vol_tail2}
		\mathbb{P}_{p'}[|\mathcal{C}(0)|\geq n] \leq \frac{\mathbb{E}_{p'}[|\mathcal{C}(0)|]}{n}= \frac{\chi(p')}{n}\leq C_1n^{-1/2}\,.
	\end{equation}
	Combining \eqref{eq:vol_tail1} and \eqref{eq:vol_tail2}, we obtain
	\begin{equation}
		\mathbb{P}_{p_c}[|\mathcal{C}(0)|\geq n]  \leq  C_1n^{-1/2} + C_0C_1^{1/2} n^{-1/4} \sqrt{\mathbb{P}_{p_c}[|\mathcal{C}(0)|\geq n]}\,,
	\end{equation}
	from which \eqref{eq:MF volume} follows readily for $A$ large enough (depending on $C_0$ and $C_1$).
	\end{proof}

\begin{Rem}
A cluster with volume of $n$ typically has diameter of order $n^{1/4}$. Since $L(p)\asymp (p_c-p)^{-1/2}$, this explains why the relevant near-critical window in the proof above is $p'=p_c - n^{-1/2}$.
\end{Rem}

	\begin{proof}[Proof of \eqref{eq:MF bound M}]
	By definition, we have
	\begin{align*}
	\begin{split}
		M(p_c,h)= \sum_{n\geq 1} \mathbb{P}_{p_c}[|\mathcal{C}(0)|= n] (1-e^{-hn}) &= \sum_{n\geq 1} \mathbb{P}_{p_c}[|\mathcal{C}(0)|\geq n] (e^{-hn}-e^{-h(n+1)}) \\
		&\leq A(1-e^{-h}) \sum_{n\geq1} n^{-1/2}e^{-hn} \\
		&\leq Ah \int_{0}^{\infty} t^{-1/2}e^{-ht} dt \\
		&= Ah h^{-1/2}  \int_{0}^{\infty} s^{-1/2}e^{-s} ds = \sqrt{\pi}A h^{1/2}\,,
	\end{split}
	\end{align*}
	where we have used (discrete) integration by parts in the first line, \eqref{eq:MF volume} and the inequality 		$1-e^{-h}\leq h$ in the second line, the fact that $f(t)=t^{-1/2}e^{-ht}$ is decreasing for $t\geq0$ in the third line, and the change of variables $s=ht$ in the fourth line. This concludes the proof.
	\end{proof}

\bibliographystyle{alpha}
\bibliography{biblioPerconote}
\end{document}